\documentclass[twocolumn,amsthm]{autart}%

\usepackage{graphics}
\usepackage{epsfig}
\usepackage{times}
\usepackage{amsmath}
\usepackage{amssymb}
\usepackage{bbm}
\usepackage{multirow}
\usepackage{array}
\usepackage{amsfonts}
\usepackage{subcaption}

\usepackage{amsthm}
\usepackage{comment}
\usepackage{balance}
\usepackage{booktabs}
\usepackage{enumerate}
\usepackage{dsfont}
\usepackage{wrapfig}
\usepackage{xspace}
\usepackage[per-mode=symbol]{siunitx}

\usepackage{pgfplots}\pgfplotsset{compat=1.9}
\usepackage{grffile}
\pgfplotsset{compat=newest}  %
\usetikzlibrary{plotmarks}
\usetikzlibrary{arrows.meta}
\usepgfplotslibrary{patchplots}
\usepgfplotslibrary{groupplots}

\usepackage{hyperref}

\usepackage[round]{natbib}%
\providecommand{\U}[1]{\protect\rule{.1in}{.1in}}
\providecommand{\U}[1]{\protect\rule{.1in}{.1in}}

\theoremstyle{definition}
\newtheorem{assumption}{Assumption}
\newtheorem{definition}{Definition}

\theoremstyle{plain}

\newtheorem{lemma}{Lemma}

\newtheorem{remark}{Remark}
\newtheorem{problem}{Problem}
\theoremstyle{definition}

\usepackage{mathtools}
\usepackage{amsmath} %
\usepackage{amssymb}  %
\usepackage{mathrsfs}
\usepackage{xcolor}
\usepackage{tabularx}
\usepackage{booktabs}           %
\usepackage{subcaption}

\usepackage{todonotes}

\newcommand{\hilitediff}[1]{#1}

\newcommand{\ie}{i.e.,\xspace}

\DeclareMathOperator*{\minimize}{minimize}

\ifdefined\mathbb

\fi

\ifdefined\mathcal

\newcommand{\NNN}{\mathcal{N}\xspace}

\fi

\ifdefined\mathfrak

\begin{document}
\begin{frontmatter}
\title{A Safety-Prioritized Receding Horizon Control Framework for Platoon Formation in a Mixed Traffic Environment} %

\author[Paestum]{A M Ishtiaque Mahbub}\ead{mahbub@udel.edu},               %
\author[Paestum]{Viet-Anh Le}\ead{vietale@udel.edu},  %
\author[Paestum]{Andreas A. Malikopoulos}\ead{andreas@udel.edu}
\address[Paestum]{Department of Mechanical Engineering, University of Delaware, 126 Spencer Lab, 130 Academy Street, Newark, DE, 19716, USA}  %

\begin{abstract}
Platoon formation with connected and automated vehicles (CAVs) in a mixed traffic environment poses significant challenges due to the presence of human-driven vehicles (HDVs) with unknown dynamics and control actions.
\hilitediff{In this paper, we develop a safety-prioritized receding horizon control framework for creating platoons of HDVs preceded by a CAV.} 
\hilitediff{Our framework ensures indirect control of the following HDVs by directly controlling the leading CAV given the safety constraints.}
\hilitediff{The framework utilizes a data-driven prediction model that is based on the recursive least squares algorithm and the constant time headway relative velocity car-following model to predict future trajectories of human-driven vehicles.}
To demonstrate the efficacy of the proposed framework, we conduct numerical simulations and provide the associated scalability, robustness, and performance analyses.
\end{abstract}

\end{frontmatter}

\section{Introduction}
\label{sec:1}

\subsection{Motivation}
The emergence of connected automated vehicles (CAVs) introduces a novel mobility paradigm that enables efficient communication and real-time computation of control actions to optimize vehicle performance, traffic efficiency, and other associated benefits; see \citet{Margiotta2011, Malikopoulos2017}. %
There is a rich body of literature that adopts optimal control approaches for coordinating CAVs to improve vehicle- and network-level performance. Detailed discussions of these efforts can be found in the review papers; see \citet{Malikopoulos2016a,Guanetti2018}. 
Recently, several efforts have been reported for real-time coordination of CAVs in different traffic scenarios such as on-ramp merging roadways \citep{Ntousakis:2016aa}, roundabouts \citep{bakibillah2019optimal}, speed reduction zones \citep{Malikopoulos2018c}, signal-free intersections \citep{Mahbub2019ACC, Malikopoulos2017,Malikopoulos2020} and corridors \citep{mahbub2020decentralized,Zhao2018ITSC}. 
However, these approaches are developed with the assumption of a 100\% CAV penetration rate which is far from being realizable in the near future; see \citet{alessandrini2015automatedmixed2060}. %
\hilitediff{Hence, addressing safe and efficient operation of CAVs in a \emph{mixed traffic environment} where CAVs and human-driven vehicles (HDVs) co-exist is necessary.}

\hilitediff{
The presence of HDVs poses significant challenges to the CAVs related to modeling and control  due to the stochastic and diverse nature of human-driving behavior.
Recently, there have been several control approaches proposed to address the motion planning and control of the CAVs in different mixed traffic scenarios, such as
model predictive control \citep{leung2020infusing,wang2022data,mahbub2022_ifac}, learning-based control \citep{wu2021flow,chalaki2020ICCA,valiente2022robustness}, game-theoretic control \citep{chandra2022gameplan,liao2021cooperative}, and socially-compatible control \citep{schwarting2019social,Le2022CDC,wang2021socially,ozkan2021socially}.
Although these approaches have demonstrated quite impressive performance, they only address the motion planning and control problem for single CAVs, therefore, cannot exploit the potential benefits of coordinating multiple CAVs to manipulate the traffic flow and eliminate stop-and-go driving.
To the best of our knowledge, coordinating multiple CAVs given the presence of HDVs still remains an open problem.
}

\hilitediff{
To develop a framework for efficiently coordinating CAVs in the mixed traffic environment, our hypothesis is that the motion of the HDVs must be indirectly controlled.
In other words, we can use a CAV to restrict the motion of its following HDVs, and thus indirectly control the HDVs. 
One approach to validate this hypothesis is to leverage the concept of vehicle platooning, where we can control a CAV to force the following HDVs to form a platoon. 
A platoon is a closely-spaced group of vehicles traveling in a controlled manner, which has potential benefits such as increasing traffic throughput and fuel economy; see \citet{alam2015_platoon_benefit}.
In this paper, in an attempt to indirectly control the following HDV trajectories, we propose a framework for platoon formation where the CAVs are controlled to compel the following HDVs to form platoons.
}

\hilitediff{Next, we provide a review of the articles related to vehicle platooning that have been reported in the literature to date.}

\subsection{Literature Review}
A significant number of research efforts have been reported in the literature that explores various methods of vehicle platooning. 
Vehicle platooning can be broadly classified into two major categories: (a) platoon formation, where individual vehicles aim at creating a previously non-existent platoon or join an already existing platoon; see \citet{karbalaieali2019dynamic,johansson2018multi,xiong2019analysis,Beaver2021Constraint-DrivenStudy, mahbub2021_platoonMixed, mahbub2022ACC}, and (b) platoon control, where vehicles within an established platoon are controlled to achieve some objectives, such as string stability, safe following gap control, and coordination; see \citet{ard2020optimizing, van2017fuel,Kumaravel:2021wi, zhao2018platoonIntersection}. 
A detailed overview of the literature on vehicle platooning can be found in some survey papers; see \citet{platoon_survey_1,bhoopalam2018planning}.

The problem of platoon formation, in general, has been widely studied considering $100 \%$ CAV penetration.
Some approaches based on model predictive control (MPC) have been reported to guarantee string stability and safety; see \citet{zheng2017platooning, dunbar2006distributed,platoon_survey_1,zheng2016distributed}. 
Such control approaches, however, cannot be applied to a mixed traffic environment with a partial CAV penetration rate due to the presence of uncontrollable HDVs.
The literature on platoon formation is sparse in the context of a mixed traffic environment.
One of the most important research directions toward developing a control framework for a mixed traffic environment has been the development of cruise control and adaptive cruise control (ACC) \citep{zheng2017platooning, sharon2017protocol}, 
where a CAV preceded by a single or a group of HDVs employs a control algorithm to optimize a given objective, e.g., improvement of fuel economy \citep{jin2017fuel_mixplatoon}, minimization of backward propagating wave \citep{hajdu2019robust}.
A variation of the ACC framework has been developed to control CAVs in a mixed traffic environment; see \citet{yuan2009trafficMixedACC, chin_mpc_acc_mixed} to tackle the HDV behavior and to ensure rear-end collision avoidance. %
Recently, some efforts have combined the concept of ACC with a vehicle-to-vehicle communication protocol and proposed connected cruise control or cooperative ACC for the CAVs traveling within a mixed traffic environment; see \citet{orosz2016connected,hajdu2019robust}.
Other approaches have employed robust or data-driven MPC to ensure the safety of the CAVs in mixed vehicle platoons; see \citet{Lan_datadriven_mpc_mixedPlatoon,feng2021robust}. 
These approaches are limited to the cases where the objective is to control the ego CAV to join and/or to maintain the stability and safety of an already formed platoon.

The performance of MPC-based controllers is highly affected by the accuracy of prediction of the HDV trajectories within a look-ahead horizon.
Several research efforts reported in the literature have considered different approaches to estimate and predict the driving behavior of HDVs. 
For example, \citet{lu2019ecological} used a variation of the car-following model for predicting HDV trajectories to design an eco-ACC controller.
However, ACC controllers using car-following models such as the intelligent driver model (IDM); see \citet{treiber2013traffic}, do not always perform well while they exhibit string stability implications that can lead to rear-end collision; see \citet{milanes2014ACC}.
\citet{milanes2013cooperative} proposed a cooperative ACC where the control parameters are derived using system identification on real-world experimental data. The issue with such an approach is that the control parameters cannot capture the instantaneous changes in HDV behavior.
\citet{naus2010mpcAcc} proposed an explicit MPC-based ACC controller that employs a prediction model considering a constant speed of the preceding vehicle within the prediction horizon and does not incorporate the complex car-following dynamics of the human drivers.
\citet{dollar2021mpc} utilized an IDM model to identify offline the human driving styles in a car-following scenario and developed an MPC-based cruise control for CAV control.
\citet{jin2018connected} proposed an optimal cruise control design in which feedback gains and driver reaction time of HDVs were estimated in real-time by a sweeping least squares method.
\citet{gong2018cooperative} developed a cooperative MPC framework and combined Newell car-following model with an online curve-matching algorithm to anticipate the response delay of the HDVs.

\subsection{\hilitediff{Contributions of This Paper}}

\hilitediff{Although the existing literature on vehicle platooning is relatively rich, most of the research efforts have concentrated on platoon formation of multiple CAVs given safety requirements with surrounding HDVs, or controlling CAVs to form and maintain a platoon with the preceding HDVs.}
\hilitediff{In contrast, our approach attempts to develop a control framework for platoon formation for a CAV with multiple following HDVs.
If we can form and maintain platoons with the following HDVs, we can optimally coordinate mixed vehicle platoons to eliminate stop-and-go driving in traffic scenarios with potential conflicts such as on-ramp merging, urban intersections, etc. (for example, see \cite{mahbub2022NHM}).} 

In earlier work, we addressed the problem of platoon formation in a mixed traffic network by considering that the leading CAV has either explicit knowledge of the following HDV dynamics; see \citet{mahbub2021_platoonMixed}, or does not have such explicit knowledge; see \citet{mahbub2022ACC}.
\hilitediff{In this paper, we propose a data-driven receding horizon control (RHC) framework that employs a prediction model for estimating the driving behavior of HDVs in real-time using a recursive least squares algorithm to predict future trajectories.} 
\hilitediff{In the proposed framework, the objectives of the CAV are (a) to form a platoon with the following HDVs, and (b) to minimize its control effort with enhanced rear-end collision safety constraints.}
To the best of our knowledge, such an approach has not yet been reported in the literature to date.

\hilitediff{In summary, the contributions of this paper are twofold:}
\hilitediff{(1) a comprehensive framework for platoon formation  to control the ego CAV that aims at forming a platoon with the following HDVs in a mixed traffic environment given the rear-end safety and system constraints (Section \ref{sec:pf}) along with a feasibility analysis (Lemmas \ref{lem:Bound_t^f} and \ref{lem:feasibility_platoon}); and and
(2) a data-driven receding horizon control approach (Section \ref{sec:model_dependent_mpc}) for platoon formation, where the driving behavior of the HDVs is estimated with the constant time headway relative velocity (CTH-RV) model and a recursive least squares algorithm.}
Finally, we provide numerical validation of the proposed approaches along with associated sensitivity, robustness, and performance analyses.

\subsection{Organization of the Paper}

The remainder of the paper proceeds as follows. 
In Section~\ref{sec:pf}, we formulate the problem of platoon formation in a mixed traffic environment and provide the modeling framework. 
\hilitediff{In Section~\ref{sec:model_dependent_mpc}, we present a data-driven predictive control framework to ensure the accuracy of HDV behavior prediction to form the platoon.}
In Section~\ref{sec:sim}, we numerically validate the effectiveness of the proposed control framework in a simulation environment. 
Finally, we provide concluding remarks in Section~\ref{sec:conc}.

\section{Problem Formulation}\label{sec:pf}
We consider a scenario where a group of vehicles, consisting of CAVs and HDVs, are traveling on a roadway as shown in Fig. \ref{fig:problem_formulation}.
We assign unique integer identities to the vehicles considered for the platoon formation problem as follows: (a) the ego CAV, which has the objective to form a platoon with its following HDVs, is indexed by $1$, (b) the preceding vehicle (PV) of CAV-$1$ is index by $0$, and (c) the HDVs following CAV-$1$ are indexed by the order of their distance from the ego CAV as $2,\ldots, N$, $N\in\mathbb{N}$ (see Fig. \ref{fig:problem_formulation}). 
\hilitediff{The objective is to control the leading CAV to form a platoon with the following HDVs that satisfies the system constraints and ensures safety in terms of rear-end collision with the preceding and following vehicles.}

Next, we define the following sets to represent different groups of vehicles. %
\begin{definition}\label{def:vehicle_sets}
The set of all vehicles considered in our problem formulation is $\mathcal{N}=\{0, 1,\ldots, N\}$. The set of HDVs following CAV-$1$ is $\mathcal{N}_{\text{HDV}}=\{2,\ldots, N\}\subset\mathcal{N}$. The set of vehicles to form the platoon is $\mathcal{N}_{p}=  \{1\} \cup \mathcal{N}_{\text{HDV}}$. 
\end{definition}

\begin{remark}
We generalize our exposition considering the existence of PV-$0$ which can be either CAV or HDV. 
In the case where PV-$0$ does not exist within a pre-defined look-ahead distance, we construct the set $\mathcal{N}$ without the element $\{0\}$ without loss of generality.
\end{remark}
\begin{remark}
For formulating a valid platoon formation problem for the vehicles in $\mathcal{N}$, the set $\mathcal{N}_{\text{HDV}}$ must be non-empty.
\end{remark}
In our formulation, we allow lane changes for HDVs in $\mathcal{N}_{\text{HDV}}$ during the platoon formation process. 
If any HDV in $\mathcal{N}_{\text{HDV}}$ decides to move to a different lane, or an HDV from an adjacent lane moves into the current lane, then we recompute the set $\mathcal{N}_{\text{HDV}}$ with updated vehicles identities. 
For example, given $\mathcal{N}_{\text{HDV}} = \{2,3,4\}$, let us consider two cases: (a) if HDV-$3$ moves to a different lane, then HDV-$4$ is updated to become HDV-$3$, resulting in $\mathcal{N}_{\text{HDV}} = \{2,3\}$, and (b) if an HDV from an adjacent lane moves in between HDV-$3$ and HDV-$4$, then the added HDV is assigned an ID of $4$, and previously known HDV-$4$ is updated to become HDV-$5$ resulting in $\mathcal{N}_{\text{HDV}} = \{2,3,4,5\}$.
\begin{figure*}
    \centering
    \includegraphics[width=0.88\textwidth,bb = 50 250 700 500,clip=true]{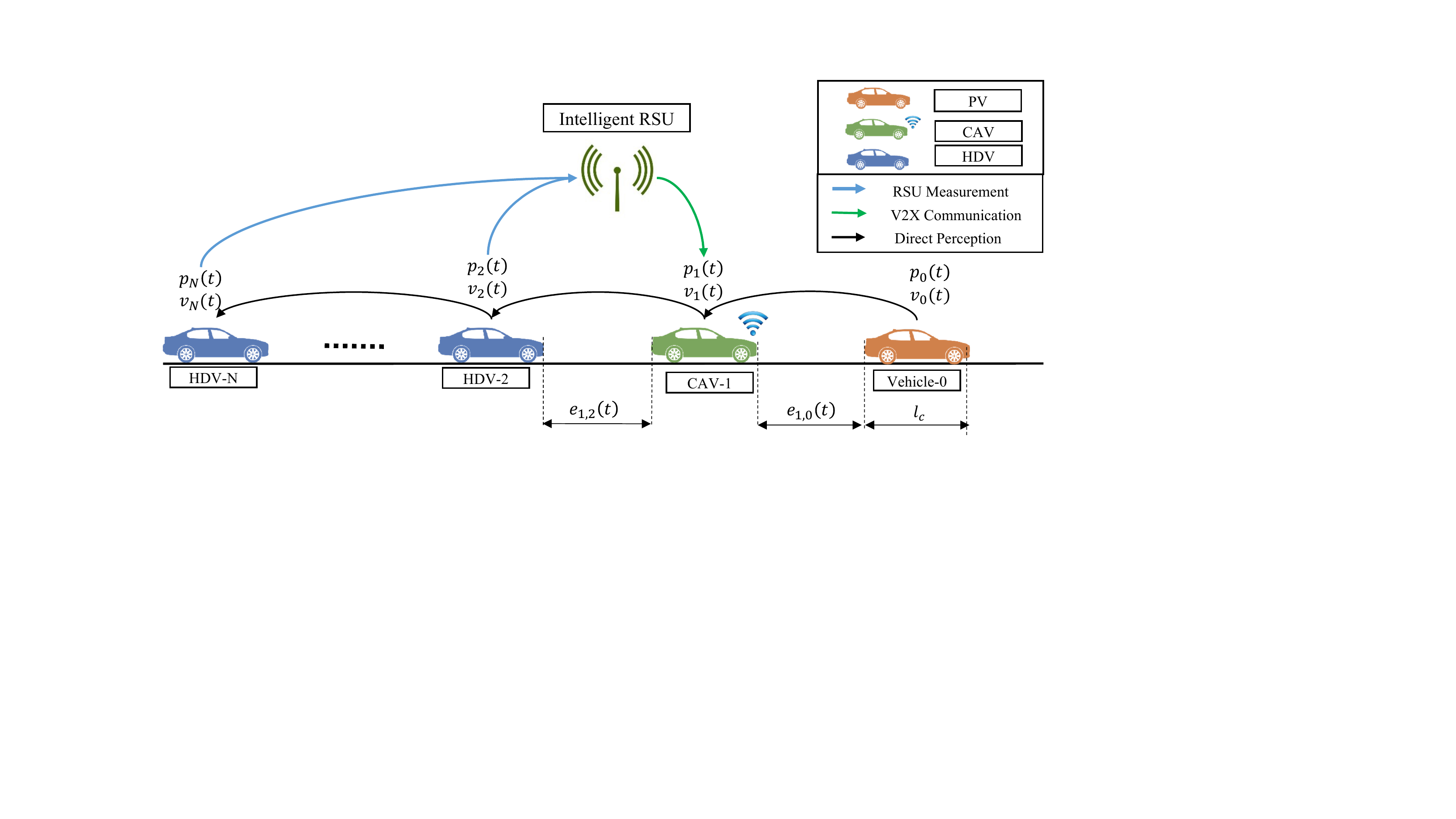}
    \caption{The ego CAV (green) is traveling with $N-1$ following HDVs (blue) and a PV (orange). 
    The communication structure is shown according to Section \ref{sec:communication_topology}. }
    \label{fig:problem_formulation}
\end{figure*}
\subsection{Vehicle Dynamics and Constraints}
We model the longitudinal dynamics of each vehicle $i\in\mathcal{N}$ as %
\begin{subequations}\label{eq:dynamics_pv}
\begin{align}
    &\dot{p}_i(t) = v_i(t), \\
    &\dot{v}_i(t) = u_i(t),
\end{align}
\end{subequations}
where $p_i(t)\in \mathcal{P}_i$, $v_i(t)\in \mathcal{V}_i$ and $u_i(t)\in\mathcal{U}_i$ are the position of the front bumper, speed and control input (acceleration/deceleration) of each vehicle $i\in \mathcal{N}$, respectively. 
The sets $\mathcal{P}_{i}$, $\mathcal{V}_{i}$, and $\mathcal{U}_{i}$, $i\in\mathcal{N},$ are complete and totally bounded subsets of $\mathbb{R}$. 

The speed $v_i(t)$ and control input $u_i(t)$ of each vehicle $i \in \mathcal{N}$ are subjected to the following constraints
\begin{subequations}\label{eq:state_control_constraints}
\begin{align}
    0\le v_{\min} \le v_i(t) &\le v_{\max},\label{eq:state_constraints}\\
    u_{\min} \le u_i(t) &\le u_{\max},\label{eq:control_constraints}
\end{align}
\end{subequations}
where $v_{\min}$ and $v_{\max}$ are the minimum and maximum allowable speed of the considered roadway, respectively, and $u_{\min}$ and $u_{\max}$ are the minimum and maximum control input, respectively. 
To simplify the \hilitediff{exposition} in the paper and without loss of generality, we consider that all the vehicles have the same attributes. 
Thus, we can consider the same minimum and maximum control input $u_{\min}$ and $u_{\max}$ for all the vehicles in \eqref{eq:control_constraints}.

To formulate the rear-end collision constraint between two consecutive vehicles $i, (i-1) \in \mathcal{N}$, we use the following definitions.
\begin{definition}\label{def:s_i}
The safe following gap $s_i(t)$ between two consecutive vehicles $i \text{ and }(i-1)\in\mathcal{N}$ is
\begin{equation}\label{eq:s_i}
    { s_i(t)= \rho_iv_i(t)+ s_0,}
\end{equation}
where $\rho_i\in \mathbb{R}_{>0}$ denotes a safe time headway that each vehicle $i\in\mathcal{N}$ maintains while following its immediate preceding vehicle $i-1\in\mathcal{N}$, and $s_0\in \mathbb{R}_{>0}$ is the standstill distance denoting the minimum bumper-to-bumper gap at stop.
\end{definition}
\begin{definition}\label{def:2}
The \textit{headway} $\Delta p_i(t)$ and \emph{approach rate} $\Delta v_i(t)$ of vehicle $i\in\mathcal{N}$ denote the bumper-to-bumper inter-vehicle spacing and speed difference, respectively, between the two consecutive vehicles $i,~(i-1) \in \mathcal{N}$, i.e.,
\begin{subequations}\label{eq:delta}
\begin{align}
    &\Delta p_i(t)=p_{i-1}(t)- p_i(t)-l_c, \label{eq:headway}\\
    &\Delta v_i(t) = v_{i-1}(t) - v_i(t), \label{eq:approach_rate}
\end{align}
\end{subequations}
where $l_c\in\mathbb{R}_{>0}$ is the length of each vehicle. 
We consider that all vehicles under consideration have the same length $l_c$ for simplicity.
\end{definition}

The rear-end collision avoidance constraint between two consecutive vehicles $i, i-1\in\mathcal{N}$ can thus be written as
\begin{align}\label{eq:rearend_constraint}
    \Delta p_i(t) \ge s_i(t).
\end{align}
The dynamics \eqref{eq:dynamics_pv} of each vehicle $i\in\mathcal{N}$ can take different forms based on the consideration of connectivity and automation.
\hilitediff{For CAV-$1$, the control input $u_1(t)$ is derived by solving an RHC problem, the structure of which we introduce and discuss in detail in Section~\ref{sec:model_dependent_mpc}.}
In contrast to the CAV, we consider a generic car-following model-based control policy of the following form to define the predecessor-follower coupled dynamics (see Fig. \ref{fig:problem_formulation}) of each HDV $i$ in $\mathcal{N}_{\text{HDV}}$,
\begin{gather}\label{eq:hdv_dynamics}
     {{u}_i(t) = f_i (\Delta p_i(t), \Delta v_i(t), v_i(t)),}
\end{gather}
where $f_i(\cdot)$ represents the behavioral model of the car-following dynamics of each HDV $i$.
There are several car-following models reported in the literature that can emulate a varied class of human driving behavior; see \cite{bando1995dynamical,treiber2013traffic}. 
For example, a widely used car-following model is the optimal velocity model (OVM) \citep{bando1995dynamical}.
One of the simplest forms of the OVM car-following model is given by \citep{bando1995dynamical}
\begin{gather}\label{eq:hdv_dynamics_ovm}
    {{u}_i(k) = \alpha_i (V_i(\delta_i(k),s_i(k)) -v_i(k)) + \beta_i \Delta v_i(t),}
\end{gather}
where $\alpha_i, \beta_i \in \mathbb{R}_{>0}, ~i\in \mathcal{N}_{\text{HDV}}$ denote the control gain representing the driver's sensitivity coefficient and the speed-dependent coefficient, respectively, $\delta_i(t)= \Delta p_i(t)-s_i(t)$, and $V_i(\delta_i(t),s_i(t))$ denotes the equilibrium speed-spacing function
\begin{gather}
V_i(\delta_i(t),s_i(t))=
\begin{array}
[c]{ll}%
 {\frac{v_d}{2}(\tanh(\delta_i(t))}{+\tanh(s_i(t))),}
\end{array}
\label{eq:V(s)}
\end{gather}
where $v_d$ is the desired speed of the roadway. 

Note that, if there is no preceding vehicle, we set $\Delta p_i(t)=\infty$ which results in $v_i(t)$ approaching the desired speed $v_d$ with the progression of time. 
The car-following model and control parameters considered in our numerical study are provided in Section \ref{sec:sim}.

Finally, PV-$0$, if it exists, can be considered to be either a CAV or HDV. 
CAV-$1$ does not know the control structure of PV-$0$ and has to guarantee rear-end safety implications. 
\subsection{Communication Structure}\label{sec:communication_topology}

CAV-$1$ is retrofitted with appropriate sensors and communication devices to estimate in real-time the state information of the vehicles in $\mathcal{N}\setminus\{1\}$. 
For example, the state information of PV-$0$ can be directly measured by the front sensors of CAV-$1$, whereas the state information of the following HDVs in $\mathcal{N}_{\text{HDV}}$ can be done using a vehicle-to-everything communication protocol and/or intelligent roadside units. %
Consequently, we can define the structure of the information available to CAV-$1$ as follows.

\begin{definition}\label{def:information-set}
The information set $\mathcal{I}(t)$ available to CAV-$1$ at time $t$ is
\begin{align}
    \mathcal{I}(t) = \{\boldsymbol{p}_{0:N}(t), \boldsymbol{v}_{0:N}(t)\},
\end{align}
where $\boldsymbol{p}_{0:N}(t)=[{p}_0(t),\ldots, {p}_N(t)]^T$ and $\boldsymbol{v}_{0:N}(t)=[{v}_0(t),\ldots, {v}_N(t)]^T$.
\end{definition}

In our modeling framework, we impose the following assumption about the communication protocol.
\begin{assumption}\label{assum:4}
The estimation and transmission of the HDVs' state information to the CAV occur without any significant delay or error.
\end{assumption}
Assumption \ref{assum:4} might be too restrictive. However, it can be relaxed as long as the noise in the measurements and/or delays is bounded. For example, we can determine upper bounds on the state uncertainties as a result of sensing or communication errors and delays, and incorporate these into more conservative safety constraints.

In what follows, we introduce the platoon formation problem.
\subsection{Platoon Formation Problem}
Conventionally, a platoon is defined as a closely-spaced group of vehicles, where each vehicle in the group is traveling with equal headway $\Delta p_i(t)$ and speed $v_i(t)$. This means that a platoon is said to be formed for a vehicle group $\mathcal{N}_p$ at some time $t=t^p$ if for each vehicle $i\in\mathcal{N}_p$,
\begin{subequations}
\begin{align}
\Delta p_i(t) &= \Delta p_{\mathrm{eq}}, \quad  t\ge t^p, \label{eq:platoon_conventional_1}\\
v_i(t) &= v_{\mathrm{eq}}, \quad  t\ge t^p \label{eq:platoon_conventional_2},
\end{align}
\end{subequations}
where, $\Delta p_{\mathrm{eq}}, v_{\mathrm{eq}}\in \mathbb{R}_{>0}$ are the equilibrium platoon headway and speed, respectively.
However, the conventional definition of platoon formation does not hold for a group of heterogeneous vehicles having different driving behavior as we would expect from a real-world scenario. 
In the problem we are addressing, each HDV $i$ in $\mathcal{N}_{\text{HDV}}$ can have different safe time headway $\rho_i$ (Definition \ref{def:2}) and behavioral function $f_i(\cdot)$. 
As a result, $\Delta p_i(t)$ for each HDV $i$ in $\mathcal{N}_{\text{HDV}}$ will converge to different equilibrium values as time $t$ progresses, violating the conditions of platoon formation in \eqref{eq:platoon_conventional_1}.
Hence, we need to revise the definition of platoon formation in the context of a heterogeneous vehicle group, as we have in our problem formulation.

\begin{definition}\label{def:platoon_formation_revision}
For a heterogeneous vehicle group $\mathcal{N}_p$, 
a platoon is formed at some time $t=t^p$ if for each vehicle $i\in\mathcal{N}_p$ the following conditions hold
\begin{align}
   \lim_{ t \to t^p} \left\| \Delta p_i(t) - s_i(t) \right\| = 0,\label{eq:platoon_revision_1}\\
   \lim_{ t \to t^p} \left\| \Delta v_i(t) \right\| = 0.\label{eq:platoon_revision_2}
\end{align}
\end{definition}

\begin{remark}\label{rem:rmse_platoonFormation}
To determine the platoon formation time $t^p$, the conditions in Definition \ref{def:platoon_formation_revision} might be too restrictive in practice. 
Therefore, we introduce the following root-mean-squared-error-based conditions to relax the conditions in Definition \ref{def:platoon_formation_revision}.
\begin{align}
    \sqrt{\sum_{i=2}^N \bigg( \Delta p_i(t) - s_i(t) \bigg)^2} &\le \epsilon_{\Delta p},&\forall t\ge t^p,\label{eq:platoon_formation_1}\\
    \sqrt{\sum_{i=1}^N \bigg( v_i(t) - \frac{\sum_{i=1}^N v_i(t)}{N} \bigg)^2} &\le \epsilon_{v},&\forall t\ge t^p,\label{eq:platoon_formation_2}
\end{align}
where $\epsilon_{\Delta p},\epsilon_{v} \in\mathbb{R}_{>0}$ are some user-defined small deviation.
\hilitediff{The conditions \eqref{eq:platoon_formation_1} and \eqref{eq:platoon_formation_2} are used to examine the platoon formation time in the simulations in Section~\ref{sec:sim}.}
\end{remark}

Next, we formalize the problem of platoon formation in a mixed traffic environment addressed in the paper as follows.
\begin{problem}\label{prob:1}
The objective of CAV-$1$ is to derive its control input $u_1(t)$ given the information set $\mathcal{I}(t)$ so that the vehicles in $\mathcal{N}_p$ form a platoon according to the Definition \ref{def:platoon_formation_revision} and satisfies the state, control, and safety constraints in \eqref{eq:state_control_constraints} and \eqref{eq:rearend_constraint}, respectively.
\end{problem}

\hilitediff{In this paper, we adopt an RHC framework to address Problem \ref{prob:1}. 
The basic principle of an RHC framework is that the current control action sequence is obtained by solving an optimization problem with a control horizon $H \in \mathbb{N} \setminus \{0\}$, and only the first input of the solved control sequence is applied. 
Then the horizon moves forward a step and the process is repeated until a final time $t^f$ is reached.}

\subsection{Feasibility of Platoon Formation}\label{subsec:feasibility_problem}
\hilitediff{If there exists a roadway of finite length $L\in \mathbb{R}_{>0}$ to form the platoon, then we need to check whether a feasible choice of final time $t^f$ leads to a feasible Problem \ref{prob:1}.}
In our previous work, we  showed that a platoon formation with the following HDVs is achieved by non-positive control input of the leading CAV, i.e., $u_1(t)\in [u_{\min},0]$; see \cite{mahbub2021_platoonMixed, mahbub2022ACC}. 
Therefore, we can consider the extremes of $[u_{\min}, 0]$ to check whether $t^f$ is feasible.
The following result provides the feasibility check of the final step $t^f$.
\begin{lemma}\label{lem:Bound_t^f}
Let \hilitediff{$t^0 = 0$} be the initial time when CAV-$1$ starts deriving and implementing its control input $u_1(t)$, $t\ge t^0$, to form a platoon with the following HDVs.
The final time $t^f$ of the RHC framework that CAV-$1$ has available to solve Problem \ref{prob:1} on a given roadway of length $L\in\mathbb{R}_{>0}$ is bounded by the following relation,
\begin{align}\label{eq:time_feasibility}
   \frac{L}{v_1(t^0)} \le t^f \le \mathds{1}_{L\le L_s} \tau_1 + (1- \mathds{1}_{L\le L_s}) \tau_2,
\end{align}
where $\mathds{1}_{L\le L_s}$ is an indicator function, $L_s = \frac{v_{\min}^2-v_1^2(t^0)}{2u_{\min}}$, $\tau_1 = \frac{-v_1(t^0) + \sqrt{v_1^2(t^0)+ 2 u_{\min}L}}{u_{\min}}$ and $\tau_2 = \frac{v_{\min}-v_1(t^0)}{u_{\min}}-\frac{v_{\min}^2-v_1^2(t^0)}{2u_{\min}v_{\min}}$.
\end{lemma}

\begin{proof}

Let $t^e$ be the time that the CAV reaches the end of the available roadway of length $L$ when cruising with a constant speed $v_1(t^0)$. 
Then, $t^e=t^0+\frac{L}{v_1(t^0)}$. 
Consequently, the minimum time that CAV-$1$ can take to traverse the distance $L$ is $ t^e-t^0$, which is then the lower bound of the horizon $t^f$ in \eqref{eq:time_feasibility}.

The maximum time that CAV-$1$ can take to travel distance $L$ can be computed by considering the following piecewise control input of CAV-$1$ constructed using the constraints in \eqref{eq:state_control_constraints},
\begin{gather}\label{eq:piecewise_control}
    u_1(t) =
    \begin{cases}
 u_{\min}, & \text{ if } v_1(t)>v_{\min}, \\
 0, & \text{ if } v_1(t)=v_{\min}.
\end{cases}
\end{gather}

Let us consider the time $t=t^s$ where the control input $u_1(t)$ switches from $u_1(t)=u_{\min}$ to $u_1(t)=0$ in \eqref{eq:piecewise_control}. 
Using \eqref{eq:dynamics_pv}, we have $v_{\min}=v_1(t^0)+u_{\min} (t^s-t^0)$, which yields $t^s=t^0 + \frac{v_{\min}-v_1(t^0)}{u_{\min}}$. 
Furthermore, using \eqref{eq:dynamics_pv}, we can compute that the control switch in \eqref{eq:piecewise_control} occurs after traveling the distance $L_s= \frac{v_{\min}^2-v_1^2(t^0)}{2u_{\min}}$. 
We need to consider the following two cases:

(a) If $L \le L_s$, then the upper bound of $t^f$ can be computed by solving $\frac{1}{2}u_{\min}(\tau_1)^2+v_1(t^0) \tau_1 - L = 0$ for $\tau_1$, which yields $\tau_1 = \frac{-v_1(t^0) + \sqrt{v_1^2(t^0)+ 2 u_{\min}L}}{u_{\min}}$. 
Here, $\tau_1$ is the upper bound of $t^f$.

(b) If $L > L_s$, then CAV-$1$ travels the distance $L_s$ with control input $u_1(t)=u_{\min}$ and time duration $t^s-t^0$, and the remaining distance $L-L_s$ with cruising speed $v_{\min}$ and time duration $\frac{L-L_s}{v_{\min}}$. 
The maximum time duration $\tau_2$ to traverse distance $L$ can be computed using $\tau_2=(t^s-t^0) + \frac{L-L_s}{v_{\min}}$. 
Using the values of $t^s, L_s$, we get, $\tau_2 = \frac{v_{\min}-v_1(t^0)}{u_{\min}}-\frac{v_{\min}^2-v_1^2(t^0)}{2u_{\min}v_{\min}}$.

Combining the above cases, we can derive the upper-bound on the final horizon as $ \mathds{1}_{L\le L_s} \tau_1 + (1- \mathds{1}_{L\le L_s}) \tau_2$, which is the right-hand term of the inequality in \eqref{eq:time_feasibility}.
\end{proof}

The conditions in \eqref{eq:time_feasibility} only provide a formal way to select an appropriate final horizon $t^f$. 
The infeasibility of the final horizon $t^f$ does not necessarily render Problem \ref{prob:1} infeasible. 
Conversely, the feasibility of $t^f$ does not imply that Problem \ref{prob:1} will be feasible as well, i.e., a platoon formation is guaranteed. 
In what follows, given that the final horizon $t^f$ is feasible according to Lemma \ref{lem:Bound_t^f}, we provide conditions to investigate whether Problem \ref{prob:1} is feasible given the constraints \eqref{eq:state_control_constraints}, and a finite roadway of length $L$ for platoon formation.

\hilitediff{Next, we check the feasibility of Problem \ref{prob:1} considering the most aggressively decelerating control structure in \eqref{eq:piecewise_control} as discussed in the following lemma.}

\begin{lemma}
\label{lem:feasibility_platoon}
Suppose that CAV-$1$ starts deriving and implementing its control input $u_1(t)$ at time $t=t^0$ to form a platoon with the following HDVs at some time $t^p \in (t^0, t^0+t^f)$ within a given roadway of length $L$, where $t^f$ is a feasible final horizon bounded by the lower- and upper-values $\tau_1$ and $\tau_2$ according to Lemma~$\ref{lem:Bound_t^f}$, respectively. 
Suppose that CAV-$1$ has the control structure as in \eqref{eq:piecewise_control}, and $L_s$ is the length where control input $u_1(t)$ switches from $u_{\min}$ to $0$.

(a) If $L_s > L$, then Problem \ref{prob:1} is feasible if
\small
\begin{equation}\label{eq:feasiblity_platoon_1}
\begin{multlined}
0 \hspace{-2pt} < \hspace{-2pt} \frac{v_N(t^0)-v_1(t^0)-\sqrt{(v_N(t^0)-v_1(t^0))^2-2u_{\min}\Delta p_s}}{u_{\min}} \hspace{-2pt} \le \hspace{-2pt} \tau_1
\end{multlined}
\end{equation}
\normalsize
holds and, (b) if $L_s \le L$, then Problem \ref{prob:1} is feasible if
\begin{align}\label{eq:feasiblity_platoon_2}
    0 < \frac{\Delta p_s +v_{\min}\tau_s - L_s}{(v_{\min}-v_N(t^0))} \le \tau_2
\end{align}
holds, where $\Delta p_s = \sum_{i=2}^N (\Delta p_i(t^0)-s_i(t^0)) $.
\end{lemma}

\begin{proof}
Consider that CAV-$1$ takes the time duration $\tau_p$ to form a platoon.

Case (a):  If $L_s > L$, we check whether a platoon can be formed with the control structure \eqref{eq:piecewise_control}. 
Suppose that $\Delta p_s = \sum_2^N (\Delta p_i(t^0)-s_i(t^0))$ is the additional spacing between CAV-$1$ and HDV-$N$ beyond the dynamic following spacing $s_i(t^0)$. 
Hence, to form a platoon with time duration $\tau_p$ according to Definition \ref{def:platoon_formation_revision}, we require $v_1(t^0)\tau_p+\frac{1}{2}u_{\min}(\tau_p)^2-v_N(t^0)\tau_p = \Delta p_s$. 
Solving this equation for $\tau_p$, we get $\tau_p = \frac{(v_N(t^0)-v_1(t^0))-\sqrt{(v_N(t^0)-v_1(t^0))^2-2u_{\min}\Delta p_s}}{u_{\min}}$. 
The value of $\tau_p$ is lower-bounded by $0$ to ensure positive value and upper-bounded by $\tau_1$ so that platoon is not formed beyond $L$, which yields \eqref{eq:feasiblity_platoon_1}.

Case (b): If $L_s \ge L$, then to form a platoon with time duration $\tau_p$, we require $L_s + v_{\min}(\tau_p-\tau_s)-v_N(t^0)\tau_p= \Delta p_s$. 
Solving for $\tau_p$, we get $\tau_p = \frac{\Delta p_s +v_{\min}\tau_s - L_s}{(v_{\min}-v_N(t^0))}$, which is lower- and upper-bounded by $0$ and $\tau_2$ to ensure platoon formation within $L$.
\end{proof}
\begin{remark}
We use the conditions in Lemma~\ref{lem:feasibility_platoon} only to investigate the feasibility of Problem~\ref{prob:1} constrained by a limited road space of length $L$ for the case where platoon formation is not possible even with the most aggressive braking maneuver of CAV-$1$. 
Satisfaction of the conditions in Lemma~\ref{lem:feasibility_platoon}, in general, does not guarantee the existence of a solution to Problem~\ref{prob:1} given the optimization criteria discussed in Problem~\ref{prob:1}.
\end{remark}

\section{\hilitediff{Data-Driven Receding Horizon Control Framework}}
\label{sec:model_dependent_mpc}
\hilitediff{In this section, we present a data-driven RHC for platoon formation where we consider a linear prediction model called constant time headway relative velocity (CTH-RV) model \citep{gunter2019modeling,wang2020onlineRLS} as a representation of human car-following behavior.} 
\hilitediff{Moreover, we use recursive least squares (RLS) method \citep{ljung1983RLS, wang2020onlineRLS} to estimate the HDVs' car-following parameters for future prediction in RHC.}
\hilitediff{The essential steps of the proposed framework can be shown in Fig. \ref{fig:control_framework} and are outlined as follows.}

\begin{enumerate}
\item \textbf{Data-driven parameter estimation:} 
At each time instant $k$, the current states $p_i(k), v_i(k)$ of each following HDV $i$ in $\mathcal{N}_{\text{HDV}}$ are communicated to CAV-$1$. 
\hilitediff{Since the exact car-following model $f_i$ of each HDV $i$ in $\mathcal{N}_{\text{HDV}}$ is unknown to CAV-$1$, it considers a CTH-RV car-following model to represent the driving behavior of each HDV and estimates the parameters of the car-following model for each HDV given online data.}
\item \textbf{Data-driven RHC problem:} \hilitediff{CAV-$1$ then uses the estimated car-following model from Step 1 to predict the future state trajectories of the following HDVs, along with considering the worst-case action of PV-$0$. 
It then derives the control input sequence by solving the RHC problem.}
Finally, CAV-$1$ applies only the first obtained control input.
\end{enumerate}

\hilitediff{In what follows, we provide a detailed exposition of the components discussed above.}%

\begin{figure*}
    \centering
    \includegraphics[width=0.88\textwidth,bb = 150 50 900 480,clip]{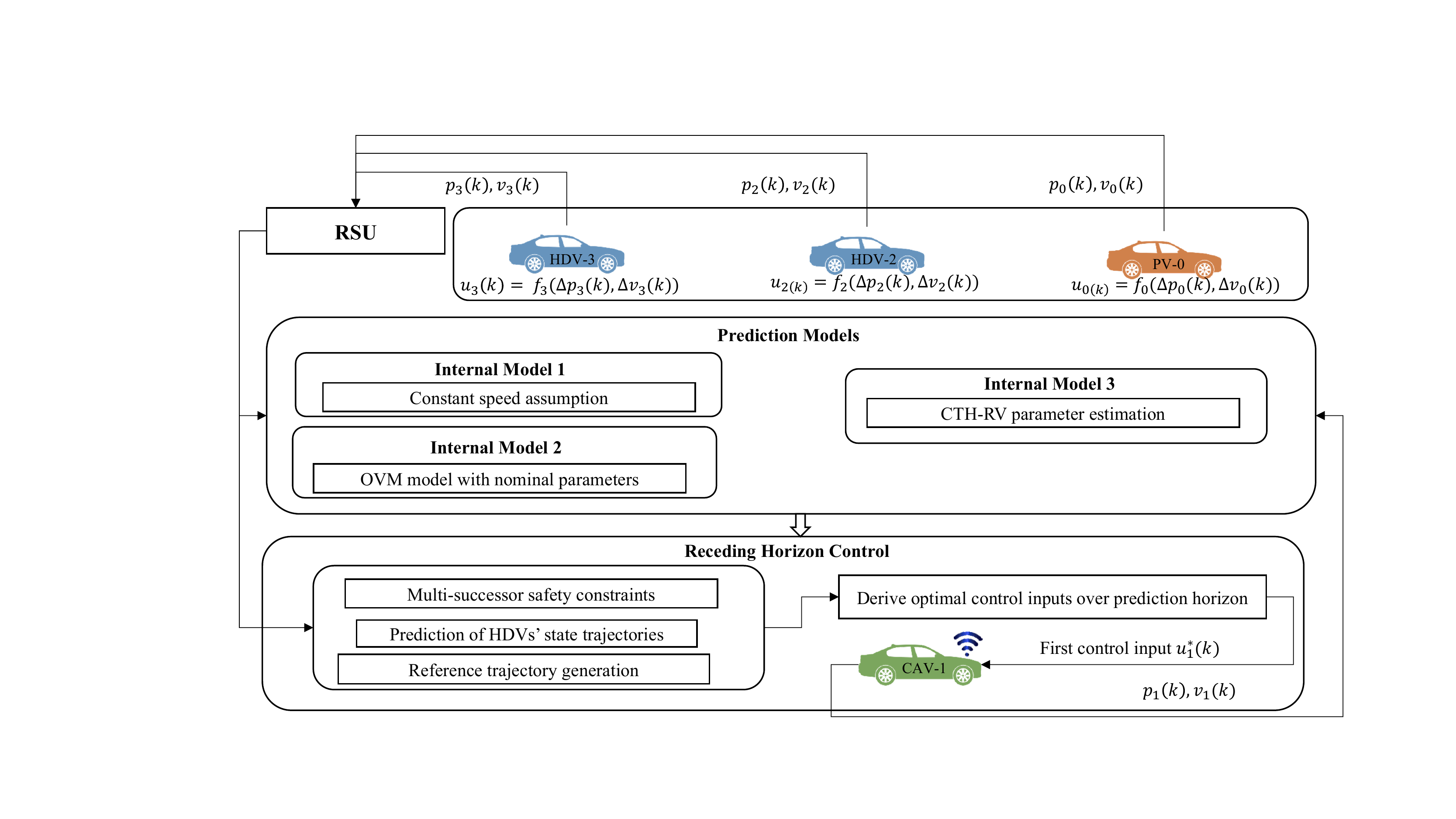}
    \caption{The structure of the proposed control framework to address Problem \ref{prob:1}.}
    \label{fig:control_framework}
\end{figure*}

\subsection{\hilitediff{Receding Horizon Control Formulation}}
\label{subsec:rhc}

\hilitediff{We consider a sampling time interval of $\tau$ to discretize the time and formulate the control problem.
Let $H \in \mathbb{N} \setminus \{0\}$ be the length of control horizon and $k \in \mathbb{N}$ be the current time step.}

The primary aim of the RHC is to minimize the squared error between the leader-follower gap $e_{1,N}(k)$ and the reference $e_{r}(k) = \sum_{i=2}^N s_i(k)$.
To this end, we formulate the first objective function $J_1$ which represents a reference tracking problem and takes the form
\hilitediff{
\begin{align}
    J_1 = \frac{1}{2} \omega_e \sum_{n=1}^{H} \left(  e_{1,N} (k+n) - e_{r}(k+n) \right )^2 ,\label{eq:cost1}
\end{align}	
where $\omega_e \in \mathbb{R}_{>0}$ is a positive weight.}
Note that using \eqref{eq:s_i} the reference output $e_{r}(k)$ can be written as
\hilitediff{\begin{align}
e_{r}(k) = \sum_{i=2}^N s_i(k) = (N-1)s_0 + \boldsymbol{\hat{\rho}}^T\boldsymbol{\hat{w}}(k),
\end{align}}
where $\boldsymbol{\hat{\rho}}=[\rho_2,\dots,\rho_N]^T$ and $\boldsymbol{\hat{w}}(k)= [v_2(k), \dots, v_N(k)]^T$. 

The second objective of the controller is to minimize the control effort of CAV-$1$ while forming the platoon. 
Thus, we have the second objective function as follows
\hilitediff{\begin{align}
    J_2 = \frac{1}{2} \omega_u \sum_{n=1}^{H} \left( u_1(k+n-1) \right)^2,\label{eq:cost2}
\end{align}
where $\omega_u \in \mathbb{R}_{>0}$ is a positive weight.}
\hilitediff{By minimizing CAV’s acceleration/deceleration, we minimize transient engine operation, leading to direct benefits in fuel consumption and emissions; see \citet{Malikopoulos2017}.}

\hilitediff{
Therefore, the RHC problem can be formulated as follows 
\begin{align}\label{eq:ocp_lin}
\centering
&\minimize_{\boldsymbol{U}_1 (k)} \quad J_1 + J_2,\\
&\text{subject to:}\nonumber\\
&\quad p_1(t+1) = p_1(t) + v_1(t) \tau + \frac{\tau^2}{2}u_1(t), \nonumber \\
&\quad v_1(t+1) = v_1(t) + u_1(t) \tau, \nonumber\\
&\quad \eqref{eq:state_constraints}, \eqref{eq:control_constraints}, \eqref{eq:rearend_constraint},\, \forall t = k+1, \dots, k+H, \; \forall i \in \NNN_p. \nonumber
\end{align}
where $\boldsymbol{U}_1(k) = [{u}_1(k),~{u}_1(k+1),\ldots,{u}_1(k+H-1)]^T$ is the vector of control inputs over the current control horizon.
}

The optimal control sequence $\boldsymbol{U}^*_1(k)$ at time instant $k$ is computed by solving the optimal control problem \eqref{eq:ocp_lin} and only the first control input is applied. 
Then the system moves to the next time instant $k+1$, and the process is repeated.

\hilitediff{
It can be observed that to solve \eqref{eq:ocp_lin} to obtain the control inputs of CAV-$1$, the states of the preceding vehicles and the following HDVs over the control horizon need to be predicted. 
As the car-following model of PV-$0$ is assumed unknown, we guarantee the rear-end safety between CAV-$1$ and PV-$0$ under the worst-case control actions of PV-$0$, \ie maximum deceleration that does not lead to the minimum allowed speed violation, which can be given by
\begin{equation}
u_0(t) = \max\, \left \{ u_{\mathrm{min}}, \frac{v_{\mathrm{min}} - v_0(t)}{\tau} \right \}. 
\end{equation}
Meanwhile, to predict the car-following behavior of the following HDVs, we utilize a data-driven prediction model that is elaborated in the next section. 
}

\subsection{Online Car-following Model Parameter Estimation}
In this section, we use a recursive least-squared formulation \citep{ljung1983RLS} to estimate the parameters of the car-following model representing the driving behavior of each of the following HDVs. 
\hilitediff{To this end, we consider the following CTH-RV model \citep{gunter2019modeling,wang2020onlineRLS},}
\begin{equation}\label{eq:cth-rv}
\begin{multlined}
    v_i(k+1) = v_i(k) + \eta_i(\Delta p_i(k) - \rho_i v_i(k))\tau + \\
    \nu_i (v_{i-1}(k)-v_i(k))\tau,
\end{multlined}
\end{equation}
where the model parameters $\eta_i$ and $\nu_i$ are the control gains on the constant time headway and the approach rate, and $\rho_i$ is the desired safe time headway for each HDV $i$ in $\mathcal{N}_{\text{HDV}}$, respectively.
We employ the linear CTH-RV model instead of other complex nonlinear models so that the resulting control problem presented in Section~\ref{subsec:rhc} is thus convex and can be solved efficiently in real-time.
Moreover, it is also observed that CTH-RV model is highly comparable to other nonlinear car-following models in terms of data fitting; see \cite{gunter2019modeling}.

Suppose that, we measure the speed $v_i(t)$, headway $\Delta p_i(t)$, and approach rate $\Delta v_i(t)$ data at a frequency corresponding to the sampling time $\tau$. 
Then we can rewrite the CTH-RV model \eqref{eq:cth-rv} for each HDV $i$ in $\mathcal{N_{\text{HDV}}}$ in discrete time as $v_i(k+1) = v_i(k) + \eta_i(\Delta p_i(k) - \rho_i v_i(k))\tau + \nu_i (v_{i-1}(k)-v_i(k))\tau$, which can be recast as
\begin{align}\label{eq:estimation_1}
    v_i(k+1) = \gamma_{i,1} v_i(k) + \gamma_{i,2} \Delta p_i(k) + \gamma_{i,3} v_{i-1}(k),
\end{align}
where $\gamma_1= (1-(\eta_i \rho_i+\nu_i)\tau)$, $\gamma_2 = \eta_i \tau$ and $\gamma_3= \nu_i \tau$ are the parameters we estimate online.
Then we can write the measurements in matrix form as
\begin{align}
    v_i(k+1) = \boldsymbol{\gamma}_i^T \boldsymbol{\phi}_i(k),
\end{align}
where $\boldsymbol{\phi}_i(k) = [v_i(k),~ \Delta p_i(k), ~v_{i-1}(k)]^T$ is the regressor vector and $\boldsymbol{\gamma}_i=[\gamma_{i,1},~\gamma_{i,2},~\gamma_{i,3}]^T$ is the parameter vector. 
If we have $H^e \in \mathbb{N} \setminus \{ 0 \}$ uniformly sampled measurements for $k= \{1, \ldots, H^e\}$, then we can estimate $\boldsymbol{\gamma}_i$ by solving the following minimization problem
\begin{align}\label{eq:least-squares}
\minimize_{\boldsymbol{\gamma}_i}\;  \frac{1}{2} \sum_{k=1}^{H^e} \xi^{(H^e-k)} [v_i(k)-\hat{v}_i(k|\boldsymbol{\gamma}_i)]^2,
\end{align}
where, $\hat{v}_i(k|\boldsymbol{\gamma}_i)=\boldsymbol{\gamma}_i^T\boldsymbol{\phi}_i(k)$ is a prediction of $v_i(k)$ based on the parameter vector $\gamma$, and $\xi\in[0,1]$ is the forgetting factor that assigns a higher weight to the recently collected data points and discounts older measurements. 
Note that, the objective function in \eqref{eq:least-squares} is quadratic in $\boldsymbol{\gamma}_i$, thus can be minimized analytically that yields
\begin{align}\label{eq:least_squares_sol}
\boldsymbol{\gamma}_i = \bigg[\sum_{k=1}^{H^e} \boldsymbol{\phi}_i(k)\boldsymbol{\phi}_i^T(k)\bigg]^{-1} \sum_{k=1}^{H^e} \boldsymbol{\phi}_i(k)v_i(k).
\end{align}

However, the above estimation procedure requires the storage of $\boldsymbol{\phi}_i(k)$ and $v_i(k)$ for all $k={1, \ldots, H^e}$, and yields the final estimated parameter vector $\boldsymbol{\gamma}_i$ for time step $H^e$. 
Since we are interested in online parameter estimation, it is computationally more efficient to update the intermediate time-dependent parameter vector $\hat{\boldsymbol{\gamma}}_i$ in \eqref{eq:least_squares_sol} recursively at each time step $k={1, \ldots, H^e}$ as new data becomes available. 
Therefore, we employ the following recursive form of \eqref{eq:least_squares_sol} known as the recursive least squares algorithm \citep{ljung1983RLS}
\small
\begin{subequations}
\begin{align}\label{eq:rls}
    &\boldsymbol{\hat{\gamma}}_i(k) = \hat{\boldsymbol{\gamma}_i}(k-1)+ \boldsymbol{L}_i(k)[v_i(k)-\hat{v}_i(k)], \\
    & \boldsymbol{P}_i(k) = \frac{1}{\xi} \bigg[\boldsymbol{P}_i(k-1) - \frac{\boldsymbol{P}_i(k-1)\boldsymbol{\phi}_i(k)\boldsymbol{\phi}_i^T(k)\boldsymbol{P}_i(k-1)}{\xi+ \boldsymbol{\phi}_i^T(k)\boldsymbol{P}_i(k-1)\boldsymbol{\phi}_i(k)}\bigg ],
\end{align}
\end{subequations}
\normalsize
\hilitediff{where
$\hat{\boldsymbol{\gamma}}_i (k)$ denotes the estimate of the parameter vector $\boldsymbol{\gamma}_i$ at time step $k$, $\boldsymbol{P}_i(k)$ is the estimation-error covariance matrix, while $\boldsymbol{L}_i(k)$ and $\hat{v}_i(k)$ can be computed as follows
\begin{subequations}
\begin{align} 
    & \hat{v}_i(k) = \hat{\boldsymbol{\gamma}_i}^T(k-1)\boldsymbol{\phi}_i(k), \\
    & \boldsymbol{L}_i(k) = \frac{\boldsymbol{P}_i(k-1)\boldsymbol{\phi}_i(k)}{\xi + \boldsymbol{\phi}_i^T(k)\boldsymbol{P}_i(k-1)\boldsymbol{\phi}_i(k)}.
\end{align}
\end{subequations}
}

The recursion of the RLS algorithm in \eqref{eq:rls} can be initiated at the time instant $k=0$ by considering an invertible matrix $\boldsymbol{P}_i(0)$ and the vector $\boldsymbol{\hat{\gamma}}_i(0)$ with some initial values.

\section{Simulation results}\label{sec:sim}
To validate the effectiveness of the control frameworks presented in the previous sections and evaluate their performance, we conduct extensive numerical simulations. 
In our analysis, we only consider feasible platoon formation problem, i.e., Problem \ref{prob:1} satisfies the feasibility requirements according to Lemmas \ref{lem:Bound_t^f} and \ref{lem:feasibility_platoon}.
Next, we discuss the configuration of the simulation environment and present an in-depth analysis of the simulation results.
\subsection{Simulation Setup}

We conducted several simulations with different numbers of following HDVs, and with or without a preceding vehicle.
To create a mixed traffic environment with different human driving styles, we employed the non-linear OVM given in \eqref{eq:hdv_dynamics_ovm} to simulate the driving behavior of the human drivers.
The parameters for each human driver's CFM were considered to be different from each other and chosen during the simulation by random perturbation of up to 30\% around nominal values.
The nominal values for the OVM car-following model are given in Table~\ref{table:cfm_param}.
We imposed a specific speed profile to be followed by the preceding vehicle so that we can analyze the robustness of the platoon formation framework under varying driving behavior.
For example, in the simulation, we considered a speed profile of the preceding vehicle that decelerates sharply to the minimum allowable speed and then sharply accelerates back to a higher speed. 
With the consideration of such an abrupt speed profile of the preceding vehicle, we investigated whether the CAV can avoid rear-end collision with the preceding vehicle.

\setlength{\tabcolsep}{10pt}
\renewcommand{\arraystretch}{1.0}

\begin{table}[!tb]
\caption{Nominal values of the car-following model}
\centering
\begin{tabular}{p{0.3\textwidth}p{0.1\textwidth}}
\toprule[1pt]%
\multicolumn{2}{c}{Optimal velocity model} \\
\midrule[0.5pt] %
Driver's sensitivity coefficient, $\alpha$ & 0.4  \\
Speed difference coefficient, $\beta$ & 0.2  \\
Desired speed, $v_d$ & 30 m/s\\
Safe time headway, $\rho$ & 1.8 s \\
\bottomrule[1pt] %
\end{tabular}
\label{table:cfm_param}
\end{table}

The parameters and weights in the control framework used for the simulations are given in Table~\ref{tab:sim-params}.
\hilitediff{We use Python for developing the simulation environment where the receding horizon control problems are formulated by CasADi \citep{CasAdi_python}.
We use the qpOASES solver \citep{ferreau2014qpoases} to solve the data-driven RHC.}
The RLS-based estimators in the data-driven RHC are initialized with the following values: 
$\hat{\boldsymbol{\gamma}}_i(0) = [0.67, 0.1, 0.18]^T$ and $\boldsymbol{P}_i(0) = 0.01 \, \mathbb{I}_3$ where $\mathbb{I}_3$ is the $3 \times 3$ identity matrix, while the forgeting factor is chosen as $\xi = 1.0$.
Note that all simulations in this work are performed on a Macbook Pro computer with a 2.7 GHz Quad-Core Intel Core i7 CPU and 16Gb RAM.

\begin{table}[!bt]
  \caption{Parameters of the receding horizon control framework}
  \label{tab:sim-params} 
  \centering
  \begin{tabular}{p{0.08\textwidth}p{0.08\textwidth}p{0.08\textwidth}p{0.08\textwidth}}
    \toprule[1pt]%
    \textbf{Parameters} & \textbf{Value} & \textbf{Parameters} & \textbf{Value} \\
    \midrule[0.5pt] %
    $\tau$ & \SI{0.1}{s} & $T_p$ & 20 \\
    $v_{\mathrm{max}}$ & \SI{35}{m/s} & $v_{\mathrm{min}}$ & \SI{0}{m/s}\\
    $u_{\mathrm{max}}$ & \SI{3}{m/s^2} & $u_{\mathrm{min}}$ & \SI{-5}{m/s^2}\\
    $\rho$ & \SI{1.5}{s} & $s_0$ & \SI{3.0}{m} \\
    $w_{e_{1,N}}$ & 1 & $w_u$ & 1 \\
    \bottomrule[1pt] %
  \end{tabular}
\end{table}

\subsection{Result Analysis and Discussion}

\subsubsection{Platoon formation and safety}
We first show the results of the platoon formation using the proposed data-driven RHC framework considering 4 following HDVs, \ie $N=5$.
We consider two different scenarios with (1) no preceding vehicle and (2) the presence of a preceding vehicle with a predefined speed trajectory. 
The trajectories of all vehicles without and with the preceding vehicle are illustrated in Figures~\ref{fig:ddMPC_a} and \ref{fig:ddMPC_b}, respectively.
The trajectories of the vehicles shown in each figure include their positions, speeds, headways, and speed gaps.
In the case where no preceding vehicle is considered, Fig.~\ref{fig:a_headway} shows that the headways of the vehicles become time invariant around $20$ s, and Fig.~\ref{fig:a_velocity} shows that the speeds of the vehicles converge to the same value. This implies that, in both scenarios, the CAV is able to form a platoon with the following HDVs. 
In addition, to challenge the safety guarantee of the CAV during the platoon formation process, we consider the presence of a preceding vehicle with an abrupt decelerating speed profile. As illustrated in Fig.~\ref{fig:ddMPC_b}, even in the presence of a preceding vehicle with an aggressive speed profile, the CAV is able to form a platoon with the following HDVs while maintaining a safe distance from the preceding vehicle.

Note that, in both of the above scenarios, none of the safety constraints in \eqref{eq:rearend_constraint} and speed constraint in \eqref{eq:state_constraints} are violated during the platoon formation process, as evident from Figures~\ref{fig:a_headway} and \ref{fig:b_headway}, and Figures~\ref{fig:a_velocity} and \ref{fig:b_velocity}, respectively. 
The rear-end safety guarantee can also be visualized by observing the non-intersecting position trajectories of the vehicles in Figures~\ref{fig:a_position} and \ref{fig:b_position}. 
This indicates the fidelity of the proposed data-driven RHC framework in satisfying all the constraints during the platoon formation process.

\subsubsection{Parameter estimation}
\hilitediff{The estimated parameters in the CTH-RV car-following model for all following HDVs are shown in Fig.~\ref{fig:estimation}.}
Recall that, CAV-$1$ characterizes in real-time the driving behaviors of four following HDVs with individual sets of estimated parameters $\eta$, $\nu$, and $\rho$. 
Initially, the estimated values of the car-following parameters show abrupt changes due to the lack of state information transmitted from the HDVs. 
However, as time progresses, more data points become available from the HDVs, and the estimation of the car-following parameters stabilizes towards the set of values that best describe the driving behavior of the HDVs. 
This is consistent with the observation by \cite{wang2020onlineRLS}, where parameters estimated using the RLS algorithm have been demonstrated to show near-convergence to the actual values.
Therefore, we can utilize the linear CTH-RV model and online RLS technique to approximate a nonlinear car-following model such as the OVM so that the resulting RHC problem is convex and thus, can be solved efficiently in real-time.

\begin{figure}[tb]
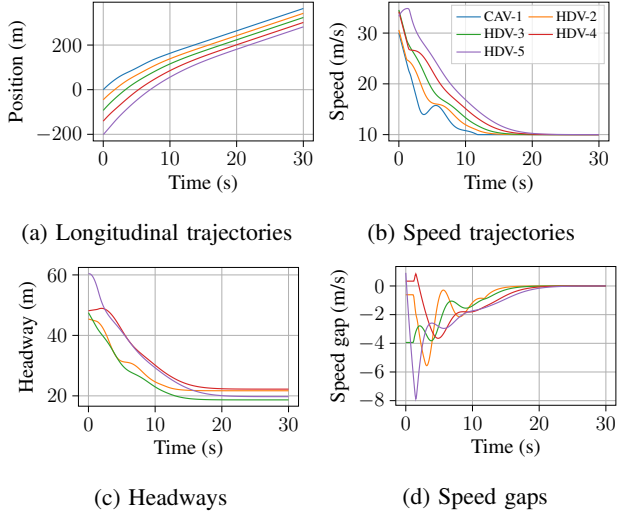

\centering
\begin{subfigure}{.23\textwidth}
    \centering
    \scalebox{0.54}{\input{tikz/ddMPC_position_a.tex}}
    \caption{Longitudinal trajectories}
    \label{fig:a_position}
\end{subfigure}
\begin{subfigure}{.23\textwidth}
    \centering
    \scalebox{0.54}{\input{tikz/ddMPC_speed_a.tex}}
    \caption{Speed trajectories}
    \label{fig:a_velocity}
\end{subfigure}
\vspace{5pt}

\begin{subfigure}{.23\textwidth}
    \centering
    \scalebox{0.54}{\input{tikz/ddMPC_headway_a.tex}}
    \caption{Headways}
    \label{fig:a_headway}
\end{subfigure}
\begin{subfigure}{.23\textwidth}
    \centering
    \scalebox{0.54}{\input{tikz/ddMPC_velogap_a.tex}}
    \caption{Speed gaps}
    \label{fig:a_vel_gap}
\end{subfigure}
\caption{Longitudinal trajectories, speed, headway and speed gaps of the vehicles for data-driven RHC in the simulation without a preceding vehicle.}
\label{fig:ddMPC_a}
\end{figure}

\begin{figure}[tb]
\centering
\begin{subfigure}{.23\textwidth}
    \centering
    \scalebox{0.54}{\input{tikz/ddMPC_position_b.tex}}
    \caption{Longitudinal trajectories}
    \label{fig:b_position}
\end{subfigure}
\begin{subfigure}{.23\textwidth}
    \centering
    \scalebox{0.54}{\input{tikz/ddMPC_speed_b.tex}}
    \caption{Speed trajectories}
    \label{fig:b_velocity}
\end{subfigure}
\vspace{5pt}

\begin{subfigure}{.23\textwidth}
    \centering
    \scalebox{0.54}{\input{tikz/ddMPC_headway_b.tex}}
    \caption{Headway}
    \label{fig:b_headway}
\end{subfigure}
\begin{subfigure}{.23\textwidth}
    \centering
    \scalebox{0.54}{\input{tikz/ddMPC_velogap_b.tex}}
    \caption{Speed gaps}
    \label{fig:b_vel_gap}
\end{subfigure}
\caption{Longitudinal trajectories, speed, headway, and speed gaps of the vehicles for data-driven RHC in the simulation with a preceding vehicle.}
\label{fig:ddMPC_b}
\end{figure}

\begin{figure}[!tb]
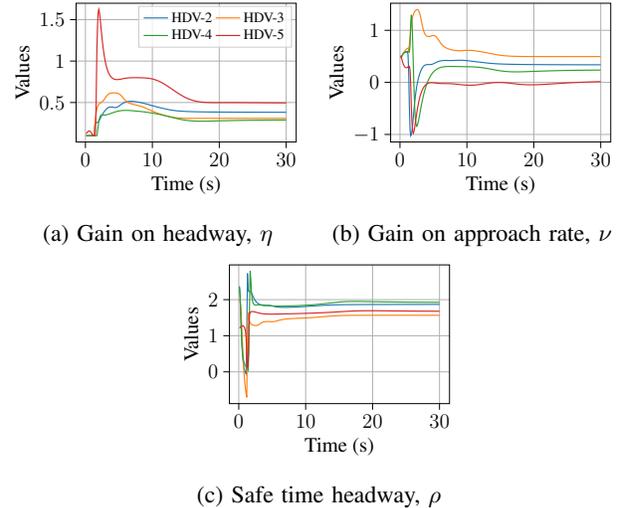

\centering
\begin{subfigure}{.23\textwidth}
    \centering
    \scalebox{0.54}{\input{tikz/eta.tex}}
    \caption{Gain on headway, $\eta$}
    \label{fig:eta}
\end{subfigure}
\begin{subfigure}{.23\textwidth}
    \centering
    \scalebox{0.54}{\input{tikz/nu.tex}}
    \caption{Gain on approach rate, $\nu$}
    \label{fig:nu}
\end{subfigure}
\vspace{5pt}

\begin{subfigure}{.23\textwidth}
    \centering
    \scalebox{0.54}{\input{tikz/rho.tex}}
    \caption{Safe time headway, $\rho$}
    \label{fig:rho}
\end{subfigure}
\caption{Estimates of the car-following parameters for all HDVs.}
\label{fig:estimation}
\end{figure}

\subsubsection{\hilitediff{Scalability and Robustness}}
For the scalability analysis, we show the position trajectories for all the vehicles in Fig.~\ref{fig:scalability} considering different numbers of following HDVs. 
Fig.~\ref{fig:scalability} verifies that the proposed control framework is able to create and maintain platoons of different sizes. 
\hilitediff{We also report the platoon formation time and the computation time of the method in those simulations considering different sizes of the platoon ($N=3$ to $N=8$) is summarized in Table~\ref{tab:metrics}.
The platoon formation time is computed using Remark~\ref{rem:rmse_platoonFormation}.
Generally, both the platoon formation time and computation time scale with the size of the vehicle platoon.
However, since the computation time of the entire algorithm is highly reasonable, the proposed framework shows promising practicality. 
}

\begin{figure}[tb]
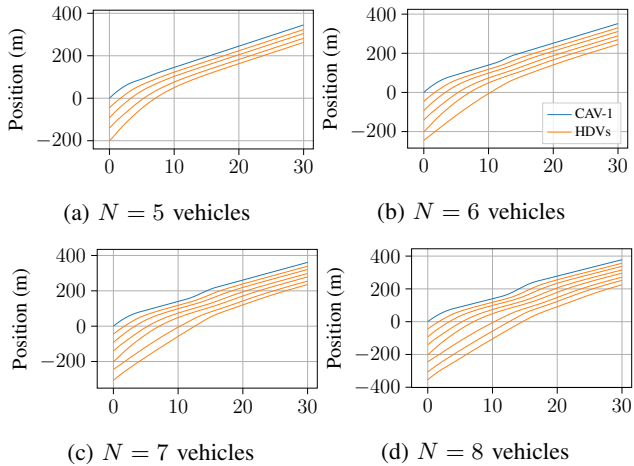

\centering
\begin{subfigure}{.23\textwidth}
    \centering
    \scalebox{0.54}{\input{tikz/scability_5.tex}}
    \caption{$N = 5$ vehicles}
\end{subfigure}
\begin{subfigure}{.23\textwidth}
    \centering
    \scalebox{0.54}{\input{tikz/scability_6.tex}}
    \caption{$N = 6$ vehicles}
\end{subfigure}
\vspace{5pt}

\begin{subfigure}{.23\textwidth}
    \centering
    \scalebox{0.54}{\input{tikz/scability_7.tex}}
    \caption{$N = 7$ vehicles}
\end{subfigure}
\begin{subfigure}{.23\textwidth}
    \centering
    \scalebox{0.54}{\input{tikz/scability_8.tex}}
    \caption{$N = 8$ vehicles}
\end{subfigure}
\caption{Longitudinal trajectories of the vehicles in the simulations with different numbers of following HDVs.}
\label{fig:scalability}
\end{figure}

\hilitediff{\begin{table}[!bt]
\vspace{13pt}
\caption{Platoon formation time and solving time for data-driven RHC in simulations with different numbers of following HDVs.}
\label{tab:metrics} 
\centering
\begin{tabularx}{0.48\textwidth}{ X | X | X }
\toprule[1pt]%
Number of vehicles & Platoon formation time (s) & Average computation time (ms) \\
\midrule[0.5pt] %
$3$ & 12.4 & 8.37 \\
$4$ & 15.3 & 6.19 \\
$5$ & 18.9 & 7.89 \\
$6$ & 23.4 & 13.31 \\
$7$ & 32.5 & 8.42 \\
$8$ & 31.6 & 10.32 \\
\bottomrule[1pt] %
\end{tabularx}
\end{table}}

\hilitediff{Finally, we examine the framework given variations of the HDV's car-following parameters, including the human driver's sensitivity coefficient $\alpha$, speed difference coefficient $\beta$, desired speed $v_d$, and safe time headway $\rho$.}
Particularly, we conduct several sets of simulations where, in each set, we consider 10 different nominal values for each OVM parameter, while keeping the other parameters constant.
Note that, the parameters for each HDV are still randomly perturbed around the nominal values.
We collect the platoon formation time for those simulations and show it in Fig.~\ref{fig:sens}.
The results suggest that with different human driving styles of the following HDVs, the control framework can guarantee the formation of a platoon. 
However, the platoon formation process gets delayed with increasing values $\beta$ and $v_d$, and is expedited with increasing values of $\alpha$ and $\rho$.

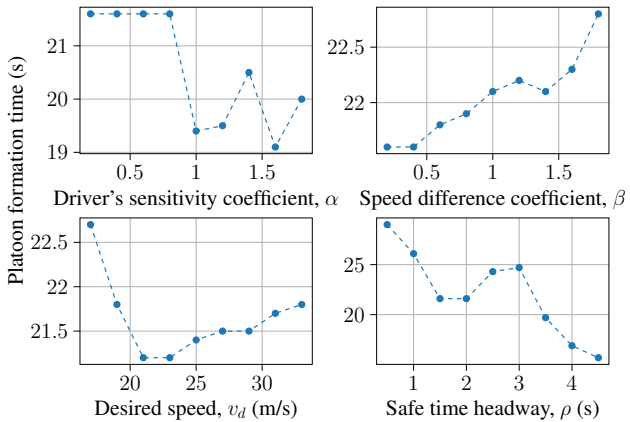
\begin{figure}[!tb]
\centering
\scalebox{0.57}{
\begin{tikzpicture}[font=\large]

\definecolor{darkgray176}{RGB}{176,176,176}
\definecolor{steelblue31119180}{RGB}{31,119,180}

\begin{groupplot}[group style={group size=2 by 2, horizontal sep = 1.5cm, vertical sep = 1.5cm}]
\nextgroupplot[
height =5.0cm,
width =7.0cm,
tick align=outside,
tick pos=left,
x grid style={darkgray176},
xlabel={ Driver’s sensitivity coefficient, $\alpha$},
xmajorgrids,
xmin=0.12, xmax=1.88,
xtick style={color=black},
y grid style={darkgray176},
ymajorgrids,
ymin=18.975, ymax=21.725,
ytick style={color=black}
]
\addplot [thick, steelblue31119180, dashed, mark=*, mark size=2, mark options={solid}]
table {%
0.2 21.6
0.4 21.6
0.6 21.6
0.8 21.6
1 19.4
1.2 19.5
1.4 20.5
1.6 19.1
1.8 20
};

\nextgroupplot[
height =5.0cm,
width =7.0cm,
tick align=outside,
tick pos=left,
x grid style={darkgray176},
xlabel={ Speed difference coefficient, $\beta$},
xmajorgrids,
xmin=0.12, xmax=1.88,
xtick style={color=black},
y grid style={darkgray176},
ymajorgrids,
ymin=21.54, ymax=22.86,
ytick style={color=black}
]
\addplot [thick, steelblue31119180, dashed, mark=*, mark size=2, mark options={solid}]
table {%
0.2 21.6
0.4 21.6
0.6 21.8
0.8 21.9
1 22.1
1.2 22.2
1.4 22.1
1.6 22.3
1.8 22.8
};

\nextgroupplot[
height =5.0cm,
width =7.0cm,
tick align=outside,
tick pos=left,
x grid style={darkgray176},
xlabel={ Desired speed, $v_d$ (m/s)},
xmajorgrids,
xmin=16.2, xmax=33.8,
xtick style={color=black},
y grid style={darkgray176},
ymajorgrids,
ymin=21.125, ymax=22.775,
ytick style={color=black}
]
\addplot [thick, steelblue31119180, dashed, mark=*, mark size=2, mark options={solid}]
table {%
17 22.7
19 21.8
21 21.2
23 21.2
25 21.4
27 21.5
29 21.5
31 21.7
33 21.8
};

\nextgroupplot[
height =5.0cm,
width =7.0cm,
tick align=outside,
tick pos=left,
x grid style={darkgray176},
xlabel={ Safe time headway, $\rho$ (s)},
xmajorgrids,
xmin=0.3, xmax=4.7,
xtick style={color=black},
y grid style={darkgray176},
ymajorgrids,
ymin=15.035, ymax=29.665,
ytick style={color=black}
]
\addplot [thick, steelblue31119180, dashed, mark=*, mark size=2, mark options={solid}]
table {%
0.5 29
1 26.1
1.5 21.6
2 21.6
2.5 24.3
3 24.7
3.5 19.7
4 16.9
4.5 15.7
};
\end{groupplot}
\draw (-1.8,-3.2) node[anchor=north west, rotate=90] {Platoon formation time (s)};
\end{tikzpicture}}
\caption{Platoon formation time under varying parameters of the OVM car-following models.}
\label{fig:sens}
\end{figure}

\section{Concluding Remarks}\label{sec:conc}
In this paper, we presented a framework to indirectly control the motion of the HDVs in a mixed traffic environment, where an ego CAV computes and implements its control input to force the following HDVs to form a platoon.  
\hilitediff{We formulated the platoon formation problem using an optimal control framework that is implemented through a receding horizon control approach subject to the state, control, and safety constraints.}
The proposed framework guarantees the rear-end collision safety of the vehicles by enforcing the multi-successor safety constraints while forming the platoon.
\hilitediff{Additionally, we developed a data-driven approach that exploits the CTH-RV car-following model and the recursive least square algorithm to estimate human driving behavior and predict human actions over a horizon.}
The efficacy of the platoon formation approaches is evaluated by extensive numerical simulations. 
\hilitediff{A direction for future research is to extend the proposed framework for optimal coordination of mixed vehicle platoons in traffic scenarios such as on-ramp merging, urban intersections, etc.}

\bibliographystyle{abbrvnat}
\bibliography{references/IDS_Publications_latest,references/acc_pt_vd_ref, references/platoon, references/mpc,references/mixed_traffic}

\end{document}